\documentclass[amstex,12pt,russian,amssymb]{article}

\usepackage{mathtext}
\usepackage[cp1251]{inputenc}
\usepackage[T2A]{fontenc}
\usepackage[russian]{babel}
\usepackage[dvips]{graphicx}
\usepackage{amsmath}
\usepackage{amssymb}
\usepackage{amsxtra}
\usepackage{latexsym}
\usepackage{ifthen}

\begin{document}

\newcounter{lemma}
\newcommand{\lemma}{\par \refstepcounter{lemma}%
{\bf Лемма \arabic{lemma}.}}

\newcounter{corollary}
\newcommand{\corollary}{\par \refstepcounter{corollary}%
{\bf Следствие \arabic{corollary}.}}

\newcounter{remark}
\newcommand{\remark}{\par \refstepcounter{remark}%
{\bf Замечание \arabic{remark}.}}

\newcounter{theorem}
\newcommand{\theorem}{\par \refstepcounter{theorem}%
{\bf Теорема \arabic{theorem}.}}

\newcounter{proposition}
\newcommand{\proposition}{\par \refstepcounter{proposition}%
{\bf Предложение \arabic{proposition}.}}

\renewcommand{\refname}{\centerline{\bf Список литературы}}

\newcommand{\proof}{{\it Доказательство.\,\,}}

\noindent УДК 517.5

{\bf Е.А.~Севостьянов} (Житомирский государственный университет им.\
И.~Франко)

\medskip
{\bf Є.О.~Севостьянов} (Житомирський державний університет ім.\
І.~Франко)

\medskip
{\bf E.A.~Sevost'yanov} (Zhitomir State University of I.~Franko)

\medskip
{\bf Об устранении изолированных особенностей классов
Орлича--Соболева с ветвлением}

{\bf Про усунення ізольованих сингулярностей класів Орліча--Соболєва
з розгалуженням}

{\bf On removability of isolated singularities of Orlicz--Sobolev
classes with bran\-ching}

\medskip
Изучается локальное поведение замкнуто-открытых дискретных
отображений классов Орлича--Соболева в ${\Bbb R}^n,$ $n\ge 3.$
Установлено, что указанные отображения $f$ имеют непрерывное
продолжение в изолированную точку $x_0$ границы области
$D\setminus\{x_0\},$ как только их внутренняя дилатация имеет
мажоранту класса $FMO$ (конечного среднего колебания) в указанной
точке и, кроме того, предельные множества отображения $f$ в $x_0$ и
на $\partial D$ не пересекаются. Другим достаточным условием
возможности непрерывного продолжения указанных отображений является
расходимость некоторого интеграла.

\medskip
Вивчається локальна поведінка замкнено-відкритих дискретних
відображень класів Орліча--Соболєва в ${\Bbb R}^n,$ $n\ge 3.$
Встановлено, що вказані відображення мають неперервне продовження до
ізольованої точки $x_0$ межі області $D\setminus\{x_0\},$ як тільки
їх внутрішня дилатація має мажоранту класу $FMO$ (скінченного
середнього коливання) у вказаній точці і, крім того, граничні
множини відображення $f$ у $x_0$ и на $\partial D$ не перетинаються.
Іншою достатньою умовою можливості неперервного продовження
зазначених відображень є розбіжність певного інтегралу.

\medskip
A local behavior of closed open discrete mappings of Orlicz--Sobolev
classes in ${\Bbb R}^n,$ $n\ge 3,$ is studied. It is proved that,
mappings mentioned above have continuous extension to isolated
boundary point $x_0$ of a domain $D\setminus\{x_0\}$ whenever $n-1$
degree of its inner dilatation has $FMO$ (finite mean oscillation)
at the point and, besides that, limit sets of $f$ at $x_0$ and
$\partial D$ are disjoint. Another sufficient condition of
possibility of continuous extension is a divergence of some
integral.

\newpage

{\bf 1. Введение.} В настоящей заметке исследуется некоторый
подкласс отображений с конечным искажением, активно изучаемых в
последнее время рядом авторов (см., напр., \cite{IM},
\cite{MRSY}--\cite{MRSY$_1$}, \cite{GRSY}, \cite{GG} и \cite{GS}).
Всюду далее $D$ -- область в ${\Bbb R}^n,$ $n\ge 2,$ $m$ -- мера
Лебега в ${\Bbb R}^n$ и ${\rm dist\,}(A,B)$ -- евклидово расстояние
между множествами $A$ и $B$ в ${\Bbb R}^n,$ $d(x, y):=|x-y|,$ $d(C)$
-- евклидов диаметр множества $C\subset{\Bbb R}^n,$
$$B(x_0, r)=\left\{x\in{\Bbb R}^n: |x-x_0|< r\right\}\,,\quad {\Bbb
B}^n := B(0, 1)\,,$$
$$S(x_0,r) = \{ x\,\in\,{\Bbb R}^n : |x-x_0|=r\}\,,\quad{\Bbb
S}^{n-1}:=S(0, 1)\,,$$
$$A(r_1,r_2,x_0)=\{ x\,\in\,{\Bbb R}^n : r_1<|x-x_0|<r_2\}\,,$$
$\omega_{n-1}$ обозначает площадь единичной сферы ${\Bbb S}^{n-1}$ в
${\Bbb R}^n,$ $\Omega_{n}$ -- объём единичного шара ${\Bbb B}^{n}$ в
${\Bbb R}^n,$ $\overline{{\Bbb R}^n}:={\Bbb R}^n\cup\{\infty\}.$ В
дальнейшем всюду символом $\Gamma(E,F,D)$ мы обозначаем семейство
всех кривых $\gamma:[a,b]\rightarrow\overline{{\Bbb R}^n},$ которые
соединяют $E$ и $F$ в $D,$ т.е. $\gamma(a)\in E,\,\gamma(b)\in F$ и
$\gamma(t)\in D$ при $t\in(a,\,b).$ Запись $f:D\rightarrow {\Bbb
R}^n$ предполагает, что отображение $f$ непрерывно в $D.$ \medskip В
дальнейшем ${\mathcal H}^k$ -- нормированная $k$-мерная мера
Хаусдорфа в ${\Bbb R}^n,$ $1\le k\le n,$ $J(x, f)={\rm det}\,
f^{\,\prime}(x)$ -- {\it якобиан отображения} $f$ в точке $x,$ где
$f^{\,\prime}(x)$ -- {\it матрица Якоби} отображения $f$ в точке
$x.$ Здесь и далее {\it предельным множеством отображения $f$
относительно множества $E\subset \overline{{\Bbb R}^n}$} называется
множество
$C(f, E):=\left\{y\in {\Bbb R}^n: \,\exists\,x_0\in E:
y=\lim\limits_{m\rightarrow \infty} f(x_m), x_m\rightarrow
x_0\right\}.$ Напомним, что отображение $f:D\rightarrow {\Bbb R}^n$
называется {\it замкнутым} (в других терминах -- {\it сохраняющим
границу отображением} (см. \cite[разд. 3, гл. II]{Vu$_1$})), если
предельное множество $C(f,
\partial D)$ отображения $f$ на границе области $D$ содержится в $\partial D^{\,\prime},$ где
$D^{\,\prime}:=f(D).$ Отображение $f:D\rightarrow {\Bbb R}^n$
называется {\it дискретным}, если прообраз $f^{-1}\left(y\right)$
каждой точки $y\in{\Bbb R}^n$ состоит только из изолированных точек.
Отображение $f:D\rightarrow {\Bbb R}^n$ называется {\it открытым},
если образ любого открытого множества $U\subset D$ является открытым
множеством в ${\Bbb R}^n.$ Пусть $U$ -- открытое множество,
$U\subset {\Bbb R}^n,$ $u:U\rightarrow {\Bbb R}$ -- некоторая
функция, $u\in L_{loc}^{\,1}(U).$ Предположим, что найдётся функция
$v\in L_{loc}^{\,1}(U),$ такая что
$\int\limits_U \frac{\partial \varphi}{\partial x_i}(x)u(x)dm(x)=
-\int\limits_U \varphi(x)v(x)dm(x)$
для любой функции $\varphi\in C_1^{\,0}(U).$ Тогда будем говорить,
что функция $v$ является {\it обобщённой производной первого порядка
функции $u$ по переменной $x_i$} и обозначать символом:
$\frac{\partial u}{\partial x_i}(x):=v.$ Функция $u\in
W_{loc}^{1,1}(U),$ если $u$ имеет обобщённые производные первого
порядка по каждой из переменных в $U,$ которые являются локально
интегрируемыми в $U.$

\medskip
Пусть $G$ -- открытое множество в ${\Bbb R}^n.$ Отображение
$f:G\rightarrow {\Bbb R}^n$ принадлежит {\it классу Соболева}
$W^{1,1}_{loc}(G),$ пишут $f\in W^{1,1}_{loc}(G),$ если все
координатные функции $f=(f_1,\ldots,f_n)$ обладают обобщёнными
частными производными первого порядка, которые локально интегрируемы
в $G$ в первой степени. Отображение $f:D\rightarrow {\Bbb R}^n$
называется {\it отображением с конечным искажением}, пишем $f\in
FD,$ если $f\in W_{loc}^{1,1}(D)$ и для некоторой функции $K(x):
D\rightarrow [1,\infty)$ выполнено условие
$\Vert f^{\,\prime}\left(x\right) \Vert^{n}\le K(x)\cdot |J(x,f)|$
при почти всех $x\in D,$ где $\Vert
f^{\,\prime}(x)\Vert=\max\limits_{h\in {\Bbb R}^n \setminus \{0\}}
\frac {|f^{\,\prime}(x)h|}{|h|}$ (см. \cite[п.~6.3, гл.~VI]{IM}.
Полагаем $l\left(f^{\,\prime}(x)\right)\,=\,\,\,\min\limits_{h\in
{\Bbb R}^n \setminus \{0\}} \frac {|f^{\,\prime}(x)h|}{|h|}.$
Отметим, что для отображений с конечным искажением корректно
определена и почти всюду конечна так называемая{\it внутренняя
дилатация $K_I(x,f)$ отображения $f$ в точке $x$}, определяемая
равенствами
\begin{equation}\label{eq0.1.1A}
K_{I}(x,f)\quad =\quad\left\{
\begin{array}{rr}
\frac{|J(x,f)|}{{l\left(f^{\,\prime}(x)\right)}^n}, & J(x,f)\ne 0,\\
1,  &  f^{\,\prime}(x)=0, \\
\infty, & \text{в\,\,остальных\,\,случаях}
\end{array}
\right.\,.
\end{equation}
Пусть $\varphi:[0,\infty)\rightarrow[0,\infty)$ -- неубывающая
функция, $f$ -- локально интегрируемая вектор-функция $n$
вещественных переменных $x_1,\ldots,x_n,$ $f=(f_1,\ldots,f_m),$
$f_i\in W_{loc}^{1,1},$ $i=1,\ldots,m.$ Будем говорить, что
$f:D\rightarrow {\Bbb R}^n$ принадлежит классу
$W^{1,\varphi}_{loc},$ пишем $f\in W^{1,\varphi}_{loc},$ если
$\int\limits_{G}\varphi\left(|\nabla f(x)|\right)\,dm(x)<\infty$ для
любой компактной подобласти $G\subset D,$ где $|\nabla
f(x)|=\sqrt{\sum\limits_{i=1}^m\sum\limits_{j=1}^n\left(\frac{\partial
f_i}{\partial x_j}\right)^2}.$ Класс $W^{1,\varphi}_{loc}$
называется классом {\it Орлича--Соболева}. Рассмотрим следующую
задачу:

\medskip
{\it пусть $x_0\in D$ и $f:D\setminus\{x_0\}\rightarrow {\Bbb R}^n$
-- отображение класса $W^{1,\varphi}_{loc}(D\setminus\{x_0\})$ с
конечным искажением, тогда при каких условиях отображение $f$ может
быть продолжено по непрерывности в точку $x_0 ?$ }

\medskip
Ответ на этот вопрос в случае, когда отображение $f$ является
гомеоморфизмом был найден нами несколько ранее (см.
\cite[теорема~5]{KRSS} и \cite[теорема~9.3]{MRSY}). Стремясь усилить
этот результат, в настоящей статье мы рассматриваем более широкий
класс замкнуто-открытых дискретных отображений. Ниже будет показано,
что для указанного класса заключение  о непрерывном продолжении в
изолированную точку границы также верно, по крайней мере, в случае
выполнения следующего дополнительного условия: $C(f, x_0)\cap C(f,
\partial D)=\varnothing.$ Разумеется, произвольные гомеоморфизмы
удовлетворяют требованиям замкнутости, дискретности, открытости, а
также указанному ограничению на предельные множества. С другой
стороны, легко указать примеры негомеоморфных замкнуто-открытых
дискретных отображений, для которых также $C(f, x_0)\cap C(f,
\partial D)=\varnothing.$ Таковым, например, является отображение с
ограниченным искажением, называемое <<закручиванием вокруг оси>> и
задаваемое в цилиндрических координатах в виде $f_m(x)=(r\cos
m\varphi, r\sin m\varphi, x_3,\ldots, x_n),$ $x=(x_1,\ldots, x_n)\in
D:={\Bbb B}^n,$ $r=|z|,$ $\varphi=\arg z,$ $z=x_1+ix_2,$ $m\in {\Bbb
N}.$ (Здесь $x_0=0$). Не лишним будет отметить, что в произвольной
меньшей области указанное отображение $f_m$ при некотором $m$ уже не
замкнуто. Скажем, это относится к области $G:=B(e_1/2, 1/2)\subset
{\Bbb B}^n,$ $e_1=(1,0,\ldots,0),$ где условие $C(f, z_0)\cap C(f,
\partial G)=\varnothing$ также может нарушаться для некоторой точки $z_m\in
G$ и больших $m.$ Другой простой пример негомеоморфного
замкнуто-открытого дискретного отображения, для которого ограничение
$C(f, x_0)\cap C(f,
\partial D)=\varnothing$ выполняется, может быть дан в виде $f(z)=z^n,$ $z\in
{\Bbb B}^2\subset {\Bbb C},$ где $x_0:=0.$

\medskip
Сформулируем главный результат настоящей заметки.

\medskip
\begin{theorem}\label{th1}
{\sl\, Пусть $n\ge 3,$ $x_0\in D,$ тогда каждое открытое, дискретное
и замкнутое ограниченное отображение $f:D\setminus\{x_0\}\rightarrow
{\Bbb R}^n$ класса $W_{loc}^{1, \varphi}(D\setminus\{x_0\})$ с
конечным искажением такое, что $C(f, x_0)\cap C(f,
\partial D)=\varnothing,$ продолжается в точку $x_0$
непрерывным образом до отображения $f:D\rightarrow {\Bbb R}^n,$ если
\begin{equation}\label{eqOS3.0a}
\int\limits_{1}^{\infty}\left[\frac{t}{\varphi(t)}\right]^
{\frac{1}{n-2}}dt<\infty
\end{equation}
и, кроме того, найдётся функция $Q\in L_{loc}^1(D),$ такая что
$K_I(x, f)\le Q(x)$ при почти всех $x\in D$ и при некотором
$\varepsilon_0>0,$ $\varepsilon_0<{\rm dist}(x_0,
\partial D),$ выполнено следующее условие расходимости интеграла:
\begin{equation}\label{eq9}
\int\limits_{0}^{\varepsilon_0}
\frac{dt}{tq_{x_0}^{\,\frac{1}{n-1}}(t)}=\infty\,.
\end{equation}
Здесь
$q_{x_0}(r):=\frac{1}{\omega_{n-1}r^{n-1}}\int\limits_{|x-x_0|=r}Q(x)\,d{\mathcal
H}^{n-1}$ обозначает среднее интегральное значение функции $Q$ над
сферой $S(x_0, r).$ В частности, заключение теоремы \ref{th1}
является верным, если $q_{x_0}(r)=\,O\left({\left[
\log{\frac{1}{r}}\right]}^{n-1}\right)$ при $r\rightarrow 0.$}
\end{theorem}

\medskip
\begin{remark}\label{rem2}
Условие (\ref{eqOS3.0a}) принадлежит Кальдерону и использовалось им
для решения задач несколько иного плана (см. \cite{Cal}).
\end{remark}

\medskip
При $n=2$ заключение теоремы \ref{th1} можно несколько усилить. Для
этой цели введём следующие обозначения. Для комплекснозначной
функции $f:D\rightarrow {\Bbb C},$ заданной в области $D\subset
{\Bbb C},$ имеющей частные производные по $x$ и $y$ при почти всех
$z=x + iy,$ полагаем $\overline{\partial} f= f_{\overline{z}} =
\left(f_x + if_y\right)/2$ и $\partial f = f_z = \left(f_x -
if_y\right)/2.$ Полагаем $\mu(z)=\mu_f(z)=f_{\overline{z}}/f_z,$ при
$f_z \ne 0$ и $\mu(z)=0$ в противном случае. Указанная
комплекснозначная функция $\mu$ называется {\it комплексной
дилатацией} отображения $f$ в точке $z.$ {\it Мак\-си\-маль\-ной
дилатацией} отображения $f$ в точке $z$ называется следующая
функция:
$K_{\mu_f}(z)\quad=\quad K_{\mu}(z)\quad=\quad\frac{1 + |\mu
(z)|}{|1 - |\mu\,(z)||}.$ Заметим, что $J(f,
z)=|f_z|^2-|f_{\overline{z}}|^2,$ где $J(f, z):={\rm
det\,}f^{\,\prime}(z),$
что может быть проверено прямым подсчётом (см., напр.,
\cite[пункт~C, гл.~I]{A}). Кроме того, заметим, что $K_I(z,
f)=K_{\mu}(z).$

\medskip
\begin{theorem}\label{th2}
{\sl\, Пусть $z_0\in D\subset {\Bbb C},$ $a, b \in {\Bbb C},$ $a\ne
b,$ тогда каждое открытое дискретное отображение
$f:D\setminus\{z_0\}\rightarrow {\Bbb C}\setminus \{a\cup b\}$
класса $W_{loc}^{1, 1}$ с конечным искажением продолжается в точку
$z_0$ непрерывным образом до отображения $f:D\rightarrow
\overline{{\Bbb C}},$ если найдётся функция $Q\in L_{loc}^1(D),$
такая что $K_{\mu}(z)\le Q(z)$ при почти всех $z\in D$ и при
некотором $\varepsilon_0>0,$ $\varepsilon_0<{\rm dist}(z_0,
\partial D),$ выполнено следующее условие расходимости интеграла
(\ref{eq9}), где $q_{z_0}(r):=\frac{1}{2\pi
r}\int\limits_{|z-z_0|=r}Q(z)\,d{\mathcal H}^{1}$ -- среднее
интегральное значение функции $Q$ над окружностью $S(z_0, r).$ В
частности, заключение теоремы \ref{th2} является верным, если
$q_{z_0}(r)=\,O\left({\left[ \log{\frac{1}{r}}\right]}\right)$ при
$r\rightarrow 0.$}
\end{theorem}

\medskip
{\bf 2. Вспомогательные сведения, основные леммы и доказательство
теоремы \ref{th1}.} Доказательство основного результата заметки
опирается на некоторый аппарат, суть которого излагается ниже (см.,
напр., \cite{MRSY}). Напомним некоторые определения, связанные с
понятием поверхности, интеграла по поверхности, а также модулей
семейств кривых и поверхностей.

\medskip Пусть $\omega$ -- открытое множество в $\overline{{\Bbb
R}^k}:={\Bbb R}^k\cup\{\infty\},$ $k=1,\ldots,n-1.$ Непрерывное
отображение $S:\omega\rightarrow{\Bbb R}^n$ будем называть {\it
$k$-мерной поверхностью} $S$ в ${\Bbb R}^n.$ Число прообразов
$N(y, S)={\rm card}\,S^{-1}(y)={\rm card}\,\{x\in\omega:S(x)=y\},\
y\in{\Bbb R}^n$ будем называть {\it функцией кратности} поверхности
$S.$ Другими словами, $N(y, S)$ -- кратность накрытия точки $y$
поверхностью $S.$ Пусть $\rho:{\Bbb R}^n\rightarrow\overline{{\Bbb
R}^+}$ -- борелевская функция, в таком случае интеграл от функции
$\rho$ по поверхности $S$ определяется равенством:  $\int\limits_S
\rho\,d{\mathcal{A}}:=\int\limits_{{\Bbb R}^n}\rho(y)\,N(y,
S)\,d{\mathcal H}^ky.$
Пусть $\Gamma$ -- семейство $k$-мерных поверхностей $S.$ Борелевскую
функцию $\rho:{\Bbb R}^n\rightarrow\overline{{\Bbb R}^+}$ будем
называть {\it допустимой} для семейства $\Gamma,$ сокр. $\rho\in{\rm
adm}\,\Gamma,$ если
\begin{equation}\label{eq8.2.6}\int\limits_S\rho^k\,d{\mathcal{A}}\ge 1\end{equation}
для каждой поверхности $S\in\Gamma.$
{\it Модулем} семейства $\Gamma$ назовём величину
$M(\Gamma)=\inf\limits_{\rho\in{\rm adm}\,\Gamma} \int\limits_{{\Bbb
R}^n}\rho^n(x)\,dm(x).$ Заметим, что модуль семейств поверхностей,
определённый таким образом, представляет собой внешнюю меру в
пространстве всех $k$-мерных поверхностей (см. \cite{Fu}). Говорят,
что некоторое свойство $P$ выполнено для {\it почти всех
поверхностей} области $D,$ если оно имеет место для всех
поверхностей, лежащих в $D,$ кроме, быть может, некоторого их
подсемейства, модуль которого равен нулю. В частности, говорят, что
некоторое свойство выполнено для {\it почти всех кривых} области
$D$, если оно имеет место для всех кривых, лежащих в $D$, кроме,
быть может, некоторого их подсемейства, модуль которого равен нулю.

Будем говорить, что измеримая по Лебегу функция $\rho:{\Bbb
R}^n\rightarrow\overline{{\Bbb R}^+}$ {\it обобщённо допустима} для
семейства $\Gamma$ $k$-мерных поверхностей $S$ в ${\Bbb R}^n,$ сокр.
$\rho\in{\rm ext}\,{\rm adm}\,\Gamma,$ если соотношение
(\ref{eq8.2.6}) выполнено для почти всех поверхностей $S$ семейства
$\Gamma.$ {\it Обобщённый модуль} $\overline M(\Gamma)$ семейства
$\Gamma$ определяется равенством
$\overline M(\Gamma)= \inf\int\limits_{{\Bbb R}^n}\rho^n(x)\,dm(x),$
где точная нижняя грань берётся по всем функциям $\rho\in{\rm
ext}\,{\rm adm}\,\Gamma.$ Очевидно, что при каждом $p\in(0,\infty),$
$k=1,\ldots,n-1,$ и каждого семейства $k$-мерных поверхностей
$\Gamma$ в ${\Bbb R}^n,$ выполнено равенство $\overline
M(\Gamma)=M(\Gamma).$

Следующий класс отображений представляет собой обобщение
квазиконформных отображений в смысле кольцевого определения по
Герингу (\cite{Ge$_3$}) и отдельно исследуется (см., напр.,
\cite[глава~9]{MRSY}). Пусть $D$ и $D^{\,\prime}$ -- заданные
области в $\overline{{\Bbb R}^n},$ $n\ge 2,$
$x_0\in\overline{D}\setminus\{\infty\}$ и $Q:D\rightarrow(0,\infty)$
-- измеримая по Лебегу функция. Будем говорить, что $f:D\rightarrow
D^{\,\prime}$ -- {\it нижнее $Q$-отображение в точке} $x_0,$ как
только
\begin{equation}\label{eq1A}
M(f(\Sigma_{\varepsilon}))\ge \inf\limits_{\rho\in{\rm
ext\,adm}\,\Sigma_{\varepsilon}}\int\limits_{D\cap
R_{\varepsilon}}\frac{\rho^n(x)}{Q(x)}\,dm(x)
\end{equation}
для каждого кольца $A(\varepsilon, r_0, x_0),$ $r_0\in(0,d_0),$
$d_0=\sup\limits_{z\in D}|z-z_0|,$
где $\Sigma_{\varepsilon}$ обозначает семейство всех пересечений
сфер $S(x_0, r)$ с областью $D,$ $r\in (0, r_0).$ Примеры таких
отображений несложно указать (см. теорему \ref{thOS4.1} ниже).

Отметим, что выражения $"$почти всех кривых$"$ и $"$почти всех
по\-вер\-х\-но\-с\-тей$"$ в отдельных случаях могут иметь две
различные интерпретации (в частности, если речь идёт о семействе
сфер, то $"$почти всех$"$ может пониматься как относительно
множества значений $r,$ так и конформного модуля семейства сфер,
рассматриваемого как частный случай семейства поверхностей).
Следующее утверждение вносит некоторую ясность между указанными
интерпретациями и может быть установлено полностью по аналогии с
\cite[лемма~9.1]{MRSY}.

\medskip
\begin{lemma}\label{lem8.2.11}{\sl\, Пусть $x_0\in D.$ Если некоторое
свойство $P$ имеет место для почти всех сфер $D(x_0, r):=S(x_0,
r)\cap D,$ где $"$почти всех$"$ понимается в смысле модуля семейств
поверхностей, то $P$ также имеет место для почти всех сфер $D(x_0,
r)$ относительно линейной меры Лебега по параметру $r\in {\Bbb R }.$
Обратно, пусть $P$ имеет место для почти всех сфер $D(x_0,
r):=S(x_0, r)\cap D$ относительно линейной меры Лебега по $r\in
{\Bbb R},$ тогда $P$ также имеет место для почти всех поверхностей
$D(x_0, r):=S(x_0, r)\cap D$ в смысле модуля семейств
поверхностей.}\end{lemma}

\medskip
Следующее утверждение облегчает проверку бесконечной серии
неравенств в (\ref{eq1A}) и может быть установлено аналогично
доказательству \cite[теорема~9.2]{MRSY}.

\medskip
\begin{lemma}\label{lem4}{\sl\,
Пусть $D,$  $D^{\,\prime}\subset\overline{{\Bbb R}^n},$
$x_0\in\overline{D}\setminus\{\infty\}$ и $Q:D\rightarrow(0,\infty)$
-- измеримая по Лебегу функция. Отображение $f:D\rightarrow
D^{\,\prime}$ является нижним $Q$-отображением в точке $x_0$ тогда и
только тогда, когда
%
$M(f(\Sigma_{\varepsilon}))\ge\int\limits_{\varepsilon}^{r_0}
\frac{dr}{||\,Q||_{n-1}(r)}\quad\forall\ \varepsilon\in(0,r_0)\,,\
r_0\in(0,d_0),$
%
где, как и выше, $\Sigma_{\varepsilon}$ обозначает семейство всех
пересечений сфер $S(x_0, r)$ с областью $D,$ $r\in (\varepsilon,
r_0),$
$ \Vert
Q\Vert_{n-1}(r)=\left(\int\limits_{D(x_0,r)}Q^{n-1}(x)\,d{\mathcal{A}}\right)^{\frac{1}{n-1}}$
-- $L_{n-1}$-норма функции $Q$ над сферой $D(x_0,r)=\{x\in D:
|x-x_0|=r\}=D\cap S(x_0,r)$.}
\end{lemma}

\medskip
Напомним, что {\it конденсатором} называют пару
$E=\left(A,\,C\right),$ где $A$ -- открытое множество в ${\Bbb
R}^n,$ а $C$ -- компактное подмножество $A.$ {\it Ёмкостью}
конденсатора $E$ называется следующая величина:
%
%
${\rm cap}\,E={\rm cap}\,\left(A,\,C\right)= \inf\limits_{u\in
W_0(E)}\,\,\int\limits_A |\nabla u(x)|^n\,\,dm(x),$
%
где $W_0(E)=W_0\left(A,\,C\right)$ -- семейство неотрицательных
непрерывных функций $u:A\rightarrow{\Bbb R}$ с компактным носителем
в $A,$ таких что $u(x)\ge 1$ при $x\in C$ и $u\in ACL.$
Здесь, как обычно, $|\nabla
u|={\left(\sum\limits_{i=1}^n\,{\left(\partial_i u\right)}^2
\right)}^{1/2}.$ Следующее утверждение имеет важное значение для
доказательства многих результатов настоящей работы (см.
\cite[предложение~10.2, гл.~II]{Ri}).

\medskip
\begin{proposition}\label{pr1*!}{\,\sl Пусть $E=(A,\,C)$ --
произвольный конденсатор в ${\Bbb R}^n$ и пусть $\Gamma_E$ --
семейство всех кривых вида $\gamma:[a,\,b)\rightarrow A$ таких, что
$\gamma(a)\in C$ и $|\gamma|\cap\left(A\setminus
F\right)\ne\varnothing$ для произвольного компакта $F\subset A.$
Тогда
${\rm cap}\,E=M(\Gamma_E).$
}
\end{proposition}

\medskip
Следующие важные сведения, касающиеся ёмкости пары множеств
относительно области, могут быть найдены в работе В.~Цимера
\cite{Zi}. Пусть $G$ -- ограниченная область в ${\Bbb R}^n$ и $C_{0}
, C_{1}$ -- непересекающиеся компактные множества, лежащие в
замыкании $G.$ Полагаем  $R=G \setminus (C_{0} \cup C_{1})$ и
$R^{\,*}=R \cup C_{0}\cup C_{1}.$ {\it Конформной ёмкостью пары
$C_{0}, C_{1}$ относительно замыкания $G$} называется величина
$C[G, C_{0}, C_{1}] = \inf \int\limits_{R} \vert \nabla u \vert^{n}\
dm(x),$
где точная нижняя грань берётся по всем функциям $u,$ непрерывным в
$R^{\,*},$ $u\in ACL(R),$ таким что $u=1$ на $C_{1}$ и $u=0$ на
$C_{0}.$ Указанные функции будем называть {\it допустимыми} для
величины $C [G, C_{0}, C_{1}].$ Мы будем говорить, что  {\it
множество $\sigma \subset {\Bbb R}^n$ разделяет $C_{0}$ и $C_{1}$ в
$R^{\,*}$}, если $\sigma \cap R$ замкнуто в $R$ и найдутся
непересекающиеся множества $A$ и $B,$ являющиеся открытыми в
$R^{\,*} \setminus \sigma,$ такие что $R^{\,*} \setminus \sigma =
A\cup B,$ $C_{0}\subset A$ и $C_{1} \subset B.$ Пусть $\Sigma$
обозначает класс всех множеств, разделяющих $C_{0}$ и $C_{1}$ в
$R^{\,*}.$ Для числа $n^{\prime} = n/(n-1)$ определим величину
%
$$\widetilde{M_{n^{\prime}}}(\Sigma)=\inf\limits_{\rho\in
\widetilde{\rm adm} \Sigma} \int\limits_{{\Bbb
R}^n}\rho^{\,n^{\prime}}dm(x)\,,$$
%
где запись $\rho\in \widetilde{\rm adm}\,\Sigma$ означает, что
$\rho$ -- неотрицательная борелевская функция в ${\Bbb R}^n$ такая,
что
%
$$\int\limits_{\sigma \cap R}\rho d{\mathcal H}^{n-1} \ge
1\quad\forall\, \sigma \in \Sigma\,.$$
%
Заметим, что согласно результата Цимера
\begin{equation}\label{eq3}
\widetilde{M_{n^{\,\prime}}}(\Sigma)=C[G , C_{0} ,
C_{1}]^{\,-1/(n-1)}\,,
\end{equation}
см. \cite[теорема~3.13]{Zi}. Заметим также, что согласно результата
Хессе
\begin{equation}\label{eq4}
M(\Gamma(E, F, D))= C[D, E, F]\,,
\end{equation}
как только $(E \cap F)\cap
\partial D = \varnothing,$
см. \cite[теорема~5.5]{Hes}.

\medskip
Напомним, что отображение $f:X\rightarrow Y$ между пространствами с
мерами $(X, \Sigma, \mu)$ и $(Y, \Sigma^{\,\prime}, \mu^{\,\prime})$
обладает {\it $N$-свой\-с\-т\-вом} (Лузина), если из условия
$\mu(S)=0$ следует, что $\mu^{\,\prime}(f(S))=0.$ Следующее
вспомогательное утверждение получено в работе \cite{KRSS} (см.
теорема 1 и следствие 2).

\medskip
\begin{proposition}\label{pr1}
{\sl\, Пусть $D$ -- область в ${\Bbb R}^n,$ $n\ge 3,$
$\varphi:(0,\infty)\rightarrow (0,\infty)$ -- неубывающая функция,
удовлетворяющая условию (\ref{eqOS3.0a}). Тогда:

1) Если $f:D\rightarrow{\Bbb R}^n$ -- непрерывное открытое
отображение класса $W^{1,\varphi}_{loc}(D),$ то $f$ имеет почти
всюду полный дифференциал в $D;$

2) Любое непрерывное отображение $f\in W^{1,\varphi}_{loc}$ обладает
$N$-свойством относительно $(n-1)$-мерной меры Хаусдорфа, более
того, локально абсолютно непрерывно на почти всех сферах $S(x_0, r)$
с центром в заданной предписанной точке $x_0\in{\Bbb R}^n$. Кроме
того, на почти всех таких сферах $S(x_0, r)$ выполнено условие
${\mathcal H}^{n-1}(f(E))=0,$ как только $|\nabla f|=0$ на множестве
$E\subset S(x_0, r).$ (Здесь $"$почти всех$"$ понимается
относительно линейной меры Лебега по параметру $r$).}

\end{proposition}

\medskip
Для отображения $f:D\,\rightarrow\,{\Bbb R}^n,$ множества $E\subset
D$ и $y\,\in\,{\Bbb R}^n,$  определим {\it функцию кратности $N(y,
f, E)$} как число прообразов точки $y$ во множестве $E,$ т.е.
\begin{equation}\label{eq1.7A}
N(y, f, E)\,=\,{\rm card}\,\left\{x\in E: f(x)=y\right\}\,,
%
N(f, E)\,=\,\sup\limits_{y\in{\Bbb R}^n}\,N(y, f, E)\,.
\end{equation}
Обозначим через $J_{n-1}f(a)$ величину, означающую $(n-1)$-мерный
якобиан отображения $f$ в точке $a$ (см. \cite[раздел~3.2.1]{Fe}).
Предположим, что отображение $f:D\rightarrow {\Bbb R}^n$
дифференцируемо в точке $x_0\in D$ и матрица Якоби
$f^{\,\prime}(x_0)$ невырождена, $J(x_0, f)={\rm
det\,}f^{\,\prime}(x_0)\ne 0.$ Тогда найдутся системы векторов
$e_1,\ldots, e_n$ и $\widetilde{e_1},\ldots,\widetilde{e_n}$ и
положительные числа $\lambda_1(x_0),\ldots,\lambda_n(x_0),$
$\lambda_1(x_0)\le\ldots\le\lambda_n(x_0),$ такие что
$f^{\,\prime}(x_0)e_i=\lambda_i(x_0)\widetilde{e_i}$ (см.
\cite[теорема~2.1 гл. I]{Re}), при этом,
\begin{equation}\label{eq11C}
|J(x_0, f)|=\lambda_1(x_0)\ldots\lambda_n(x_0),\quad \Vert
f^{\,\prime}(x_0)\Vert =\lambda_n(x_0)\,, \quad
l(f^{\,\prime}(x))=\lambda_1(x_0)\,,\end{equation}
\begin{equation}\label{eq41}
K_I(x_0,
f)=\frac{\lambda_1(x_0)\cdots\lambda_n(x_0)}{\lambda^n_1(x_0)}\,,
\end{equation}
см. \cite[соотношение~(2.5), разд.~2.1, гл.~I]{Re}. Числа
$\lambda_1(x_0),\ldots\lambda_n(x_0)$ называются {\it главными
значениями}, а вектора $e_1,\ldots, e_n$ и
$\widetilde{e_1},\ldots,\widetilde{e_n}$ -- {\it главными векторами
} отображения $f^{\,\prime}(x_0).$ Из геометрического смысла
$(n-1)$-мерного якобиана, а также первого соотношения в
(\ref{eq11C}) вытекает, что
\begin{equation}\label{eq10C}
\lambda_1(x_0)\cdots\lambda_{n-1}(x_0)\le J_{n-1}f(x_0)\le
\lambda_2(x_0)\cdots\lambda_n(x_0)\,,
\end{equation}
в частности, из (\ref{eq10C}) следует, что $J_{n-1}f(x_0)$
положителен во всех тех точках $x_0,$ где положителен якобиан
$J(x_0, f).$

\medskip
Следующие два утверждения несут в себе основную смысловую нагрузку
данной заметки. Первое из них впервые установлено для случая
гомеоморфизмов в работе \cite{KR$_1$} (см. теорему 2.1).

\medskip
\begin{lemma}{}\label{thOS4.1} {\sl Пусть $D$ -- область в ${\Bbb R}^n,$
$n\ge 3,$ $\varphi:(0,\infty)\rightarrow (0,\infty)$ -- неубывающая
функция, удовлетворяющая условию (\ref{eqOS3.0a}).
Если $n\ge 3,$ то каждое открытое дискретное отображение
$f:D\rightarrow {\Bbb R}^n$ с конечным искажением класса
$W^{1,\varphi}_{loc}$ такое, что $N(f, D)<\infty,$ является нижним
$Q$-отображением в каждой точке $x_0\in\overline{D}$ при $Q(x)=N(f,
D)\cdot K^{\frac{1}{n-1}}_I(x, f),$ где внутренняя дилатация $K_I(x,
f)$ отображения $f$ в точке $x$ определена соотношением
(\ref{eq0.1.1A}), а кратность $N(f, D)$ определена вторым
соотношением в (\ref{eq1.7A}).}
\end{lemma}

\medskip
\begin{proof}
Заметим, что $f$ дифференцируемо почти всюду ввиду предложения
\ref{pr1}. Пусть $B$ -- борелево множество всех точек $x\in D,$ в
которых $f$ имеет полный дифференциал $f^{\,\prime}(x)$ и $J(x,
f)\ne 0.$ Применяя теорему Кирсбрауна и свойство единственности
аппроксимативного дифференциала (см. \cite[пункты~2.10.43 и
3.1.2]{Fe}), мы видим, что множество $B$ представляет собой не более
чем счётное объединение борелевских множеств $B_l,$
$l=1,2,\ldots\,,$ таких, что сужения $f_l=f|_{B_l}$ являются
билипшецевыми гомеоморфизмами (см., напр., \cite[пункты~3.2.2, 3.1.4
и 3.1.8]{Fe}). Без ограничения общности, мы можем полагать, что
множества $B_l$ попарно не пересекаются. Обозначим также символом
$B_*$ множество всех точек $x\in D,$ в которых $f$ имеет полный
дифференциал, однако, $f^{\,\prime}(x)=0.$

\medskip
Ввиду построения, множество $B_0:=D\setminus \left(B\bigcup
B_*\right)$ имеет лебегову меру нуль. Следовательно, по
\cite[теорема~9.1]{MRSY}, ${\mathcal H}^{n-1}(B_0\cap S_r)=0$ для
почти всех сфер $S_r:=S(x_0,r)$ с центром в точке
$x_0\in\overline{D},$ где $"$почти всех$"$ следует понимать в смысле
конформного модуля семейств поверхностей. По лемме \ref{lem8.2.11}
также ${\mathcal H}^{n-1}(B_0\cap S_r)=0$ при почти всех $r\in {\Bbb
R}.$

\medskip
По предложению \ref{pr1} и из условия ${\mathcal H}^{n-1}(B_0\cap
S_r)=0$ для почти всех $r\in {\Bbb R}$ вытекает, что ${\mathcal
H}^{n-1}(f(B_0\cap S_r))=0$ для почти всех $r\in {\Bbb R}.$ По этому
предложению также ${\mathcal H}^{n-1}(f(B_*\cap S_r))=0,$ поскольку
$f$ -- отображение с конечным искажением и, значит, $\nabla f=0$
почти всюду, где $J(x, f)=0.$

\medskip
Пусть $\Gamma$ -- семейство всех пересечений сфер $S_r,$
$r\in(\varepsilon, r_0),$ $r_0<d_0=\sup\limits_{x\in D}\,|x-x_0|,$
с областью $D.$ Для заданной функции $\rho_*\in{\rm
adm}\,f(\Gamma),$ $\rho_*\equiv0$ вне $f(D),$ полагаем $\rho\equiv
0$ вне $D$ и на $B_0,$
$$\rho(x)\ \colon=\ \rho_*(f(x))\left(|J(x, f)|\cdot
K_I^{\frac{1}{n-1}}(x, f)\right)^{\frac{1}{n}} \qquad\text{при}\
x\in D\setminus B_0\,.$$
Учитывая соотношения (\ref{eq11C}) и (\ref{eq10C}),
\begin{equation}\label{eq12C}
\rho(x)\ge\rho_*(f(x))(J_{n-1}f(x))^{\frac{1}{n-1}}\,.
\end{equation}
Пусть $D_{r}^{\,*}\in f(\Gamma),$ $D_{r}^{\,*}=f(D\cap S_r).$
Заметим, что
$D_{r}^{\,*}=\bigcup\limits_{i=0}^{\infty} f(S_r\cap B_i)\bigcup
f(S_r\cap B_*)$
и, следовательно, для почти всех $r\in (0, r_0)$
\begin{equation}\label{eq10B}
1\le \int\limits_{D^{\,*}_r}\rho^{n-1}_*(y)d{\mathcal A_*} \le
\sum\limits_{i=0}^{\infty} \int \limits_{f(S_r\cap B_i)}
\rho^{n-1}_*(y)N (y, S_r\cap B_i)d{\mathcal H}^{n-1}y +
\end{equation}
$$+\int\limits_{f(S_r\cap B_*)} \rho^{n-1}_*(y)N (y, S_r\cap B_*
)d{\mathcal H}^{n-1}y\,.$$ Учитывая доказанное выше, из
(\ref{eq10B}) мы получаем, что
\begin{equation}\label{eq11B}
1\le \int\limits_{D^{\,*}_r}\rho^{n-1}_*(y)d{\mathcal A_*} \le
\sum\limits_{i=1}^{\infty} \int \limits_{f(S_r\cap B_i)}
\rho^{n-1}_*(y)N (y, S_r\cap B_i)d{\mathcal H}^{n-1}y
\end{equation}
для почти всех $r\in (0, r_0).$
Рассуждая покусочно на $B_i,$ $i=1,2,\ldots,$ ввиду \cite[1.7.6 и
теорема~3.2.5]{Fe} и (\ref{eq12C}) мы получаем, что
$$\int\limits_{B_i\cap S_r}\rho^{n-1}\,d{\mathcal A}=
\int\limits_{B_i\cap S_r}\rho_*^{n-1}(f(x))\left(|J(x, f)|\cdot
K_I^{\frac{1}{n-1}}(x, f)\right)^{\frac{n-1}{n}}\,d{\mathcal A}=$$
$$=\int\limits_{B_i\cap S_r}\rho_*^{n-1}(f(x))\cdot\frac{\left(|J(x, f)|\cdot
K_I^{\frac{1}{n-1}}(x, f)\right)^{\frac{n-1}{n}}}{J_{n-1}f(x)}\cdot
J_{n-1}f(x)\,d{\mathcal A}\ge \int\limits_{B_i\cap
S_r}\rho_*^{n-1}(f(x))\cdot J_{n-1}f(x)\,d{\mathcal A}=$$
\begin{equation}\label{eq12B}
=\int\limits_{f(B_i\cap S_r)}\rho_{*}^{n-1}\,N(y, S_r\cap
B_i)d{\mathcal H}^{n-1}y
\end{equation} для почти всех $r\in (0, r_0).$
Из (\ref{eq11B}) и (\ref{eq12B}) вытекает, что
$\rho\in{\rm{ext\,adm}}\,\Gamma.$

Замена переменных на каждом $B_l,$ $l=1,2,\ldots\,,$ (см., напр.,
\cite[теорема~3.2.5]{Fe}) и свойство счётной аддитивности интеграла
приводят к оценке
$\int\limits_{D}\frac{\rho^n(x)}{K^{\frac{1}{n-1}}_I(x,
f)}\,dm(x)\le \int\limits_{f(D)}N(f, D)\cdot \rho^{\,n}_*(y)\,
dm(y),$ что и завершает доказательство.
\end{proof}$\Box$

\medskip
\begin{remark}\label{rem1}
Заключение леммы \ref{thOS4.1} при $n=2$ остаётся справедливым для
классов Соболева $W_{loc}^{1, 1}$ при аналогичных условиях, за
исключением дополнительного условия Кальдерона (\ref{eqOS3.0a}).
Чтобы в этом убедиться, необходимо повторить доказательство этой
леммы при $n=2,$ где необходимо учесть наличие $N$-свойства
указанных отображений на почти всех окружностях, что обеспечивается
свойством $ACL$ для произвольных классов Соболева (см.
\cite[теорема~1, п.~1.1.3, $\S$~1.1, гл.~I]{Maz}).
\end{remark}

\medskip
Имеет место следующее утверждение (см. \cite[лемма~3.11]{MRV$_2$} и
\cite[лемма~2.6, гл.~III]{Ri}).

\medskip
\begin{proposition}\label{pr3*!} {\sl\, Для каждого $a > 0$ существует
положительное число $\delta > 0$ такое, что
${\rm cap}\,\left({\Bbb B}^n,\,C\right)\ge \delta,$
%
где $C$ -- произвольный континуум в ${\Bbb B}^n$ такой что $d(C)\ge
a.$}
\end{proposition}

\medskip
Аналог следующей леммы в случае гомеоморфизмов доказан в монографии
\cite[теорема~9.3]{MRSY} (см. также работу \cite[теорема~4.1]{KR}).

\medskip
\begin{lemma}\label{lem1}
{\sl\, Пусть $n\ge 2,$ $x_0\in D$ и $Q:D\rightarrow (0, \infty)$ -
измеримая по Лебегу функция такая, что при некотором
$\varepsilon_0>0,$ $\varepsilon_0<{\rm dist}(x_0,
\partial D),$ выполнено условие расходимости интеграла (\ref{eq9}). Тогда каждое ограниченное
открытое, дискретное и замкнутое в области $D\setminus\{x_0\}$
нижнее $Q$-отображение $f:D\setminus\{x_0\}\rightarrow {\Bbb R}^n$
продолжается в точку $x_0$ непрерывным образом до отображения
$f:D\rightarrow {\Bbb R}^n,$ если $C(f, x_0)\cap C(f,
\partial D)=\varnothing.$}
\end{lemma}

\medskip
\begin{proof}
Не ограничивая общности рассуждений, можно считать, что $x_0=0$ и
$\overline{f(D\setminus\{0\})}\subset {\Bbb B}^n.$ Предположим
противное, а именно, что отображение $f$ не может быть продолжено по
непрерывности в точку $x_0=0.$ Тогда найдутся две последовательности
$x_j$ и $x_j^{\,\prime},$ принадлежащие
$D\setminus\left\{0\right\},$ $x_j\rightarrow 0,\quad
x_j^{\,\prime}\rightarrow 0,$ такие, что
$|f(x_j)-f(x_j^{\,\prime})|\ge a>0$ для всех $j\in {\Bbb N}.$
Можно считать, что $x_j$ и $x_j^{\,\prime}$ лежат внутри шара $B(0,
r_0),$ $r_0:={\rm dist\,}(0, \partial D).$ Полагаем
$r_j=\max{\left\{|x_j|,\,|x_j^{\,\prime}|\right\}},
l_j=\min{\left\{|x_j|,\,|x_j^{\,\prime}|\right\}}.$
Соединим точки $x_j$ и $x_j^{\,\prime}$ замкнутой кривой, лежащей в
$\overline{B(0, r_j)}\setminus\left\{0\right\}.$ Обозначим эту
кривую символом $C_j$ и рассмотрим конденсатор
$E_j=\left(D\setminus\left\{0\right\}\,,C_j\right).$ В силу
открытости и непрерывности отображения $f,$ пара $f(E_j)$ также
является конденсатором. Поскольку $f$ -- открытое и замкнутое
отображение, $\partial f(D\setminus\{0\})= C(f, \partial D)\cup C(f,
0).$

\medskip
Рассмотрим при $r_j<r<r_0$ проколотый шар $G_1:=B(0,
r)\setminus\{0\}.$ Заметим, что $C_j$ -- компактное подмножество
$G_1,$ тогда $f(C_j)$ -- компактное подмножество $f(G_1).$

\medskip
Ввиду открытости $f$ имеет место включение $\partial f(G_1)\subset
C(f, 0)\cup f(S(0, r)),$ откуда ввиду замкнутости и открытости
отображения $f$ множество $\partial f(G_1)\setminus C(f, 0)$
является замкнутым в ${\Bbb R}^n.$

\medskip
Отсюда вытекает, что множество $\sigma:=\partial f(G_1)\setminus
C(f, 0)$ отделяет $f(C_j)$ от $C(f, \partial D)$ в
$f(D\setminus\{0\})\cup C(f, \partial D).$ Действительно,
$$f(D\setminus\{0\})\cup C(f, \partial D)=f(G_1)\cup\sigma\cup \left(f(D\setminus\{0\})\cup C(f, \partial D)\setminus
\overline{f(G_1)}\right)\,,$$
каждое из множеств $A:=f(G_1)$ и $B:=f(D\setminus\{0\})\cup C(f,
\partial D)\setminus \overline{f(G_1)}$ открыто в топологии
пространства $f(D\setminus\{0\})\cup C(f, \partial D),$ $A\cap
B=\varnothing,$ $C_0:=f(C_j)\subset A$ и $C_1:=C(f, \partial
D)\subset B.$

\medskip
Поскольку $\sigma\subset f(S(0, r)),$ ввиду (\ref{eq3}) и
(\ref{eq4})
\begin{equation}\label{eq5C}
M(\Gamma(f(C_j), C(f, \partial D), f(D\setminus \{0\})))\le
\frac{1}{M^{n-1}(f(\Sigma_{r}))}\,,
\end{equation}
где $\Sigma_{r}$ -- семейство сфер $S(0, r),$ $r\in (r_j, r_0).$
С другой стороны, из леммы \ref{lem4} и условия расходимости
интеграла (\ref{eq9}) вытекает, что
$M^{n-1}(f(\Sigma_{r})\rightarrow\infty$ при $j\rightarrow \infty.$
В таком случае, из (\ref{eq5C}) следует, что при $j\rightarrow
\infty$
\begin{equation}\label{eq6C}
M(\Gamma(C(f, D), f(C_j), f(D\setminus\{0\})))\rightarrow 0\,.
\end{equation}
Аналогичную процедуру проделаем относительно предельного множества
$C(f, 0).$ Именно, заметим, что $C_j$ -- компакт в $G_2:=D\setminus
\overline{B(0, \varepsilon)}$ для произвольного $\varepsilon\in (0,
l_j).$ Тогда ввиду непрерывности $f$ множество $f(C_j)$ является
компактным подмножеством $f(G_{2})=f(D\setminus \overline{B(0,
\varepsilon)})$ и, в частности, $\partial f(D\setminus
\overline{B(0, \varepsilon)})\cap f(C_j)=\varnothing.$ Далее,
заметим, что $\partial f(D\setminus\overline{B(0,
\varepsilon)})\subset C(f,
\partial D)\cup f(S(0, \varepsilon)).$ Полагаем
$\theta:=\partial f(G_{2})\setminus C(f, \partial D)$ и заметим, что
$\theta$ является замкнутым, поскольку $\partial f(G_2)\subset
f(S(0, \varepsilon))\cup C(f, \partial D)$ и $C(f, \partial D)\cap
f(S(0, \varepsilon))=\varnothing$ ввиду замкнутости отображения $f$
в $D\setminus\{0\}.$ Кроме того, заметим, что $\theta$ отделяет
$C_3:=f(C_j)$ и $C_4:=C(f, 0)$ в $f(D\setminus\{0\})\cup C(f, 0).$
Действительно,
$$f(D\setminus\{0\})\cup C(f, 0)= f(G_2)\cup \theta\cup
\left(f(D\setminus\{0\})\cup C(f,
0)\setminus\overline{f(G_2)}\right)\,,$$
$A=f(G_2)$ и $B=\left(f(D\setminus\{0\})\cup C(f,
0)\setminus\overline{f(G_2)}\right)$ открыты в топологии
пространства $f(D\setminus\{0\})\cup C(f, 0),$  $A\cap
B=\varnothing,$ $C_3:=f(C_j)\in A$ и $C_4:=C(f, 0)\in B.$

\medskip
Так как $\theta\subset f(S(0, \varepsilon)),$ ввиду (\ref{eq3}) и
(\ref{eq4}) получаем:
\begin{equation}\label{eq7C}
M(\Gamma(f(C_j), C(f, 0), f(D\setminus \{0\})))\le
\frac{1}{M^{n-1}(f(\Theta_{\varepsilon}))}\,,
\end{equation}
где $\Theta_{\varepsilon}$ -- семейство сфер $S(0, \varepsilon),$
$\varepsilon\in (0, l_j).$
С другой стороны, из леммы \ref{lem4} и условия расходимости
интеграла (\ref{eq9}) вытекает, что
$M^{n-1}(f(\Theta_{\varepsilon})=\infty.$ В таком случае, из
(\ref{eq7C}) следует, что
\begin{equation}\label{eq8C}
M(\Gamma(C(f, 0), f(C_j), f(D\setminus\{0\})))=0\,.
\end{equation}
Заметим, что ввиду предложения \ref{pr1*!} и полуаддитивности модуля
смейств кривых (см. \cite[разд.~6, гл.~I]{Va}), при $j\rightarrow
\infty$ из (\ref{eq6C}) и (\ref{eq8C}) вытекает, что
$${\rm cap\,}f(E_j)\le$$
\begin{equation}\label{eq9C}
\le M(\Gamma(C(f, 0), f(C_j), f(D\setminus\{0\})))+ M(\Gamma(C(f,
\partial D), f(C_j), f(D\setminus\{0\})))\rightarrow 0\,.
\end{equation}
С другой стороны, по предложению \ref{pr3*!} ${\rm cap\,}f(E_j)\ge
\delta>0$ при всех натуральных $j,$ что противоречит (\ref{eq9C}).
Лемма доказана.
 \end{proof} $\Box$

 \medskip
{\it Доказательство теоремы \ref{th1}} вытекает из лемм
\ref{thOS4.1} и \ref{lem1}, а также того факта, что максимальная
кратность $N(f, D)$ замкнутого открытого дискретного отображения $f$
конечна (см., напр., \cite[лемма~3.3]{MS}). $\Box$

\medskip
Теперь отдельно исследуем случай $n=2.$ Для этой цели напомним, что
отображение $f:D\rightarrow \overline{{\Bbb R}^n}$ называется {\it
кольцевым $Q$-отоб\-ра\-же\-нием в точке $x_0\,\in\,D$} (см.
\cite{MRSY}--\cite{MRSY$_1$}), если соотношение
%
$M\left(f\left(\Gamma\left(S_1,\,S_2,\,A\right)\right)\right)\ \le
\int\limits_{A} Q(x)\cdot \eta^n(|x-x_0|)\ dm(x)$ 
выполнено для любого кольца $A=A(r_1,r_2, x_0),$\, $0<r_1<r_2<
r_0:={\rm dist\,}(x_0, \partial D),$ и для каждой измеримой функции
$\eta : (r_1,r_2)\rightarrow [0,\infty ]\,$ такой, что
%
%
%
$\int\limits_{r_1}^{r_2}\eta(r)dr\ge 1.$ Отметим, что кольцевые
$Q$-гомеоморфизмы продолжаются по непрерывности в изолированные
граничные точки, причём продолженное отображение также является
гомеоморфизмом (см. \cite[лемма~4 и теорема~4]{Lom}).

\medskip
{\it Доказательство теоремы \ref{th2}.} Пусть $f$ -- отображение из
условия теоремы, тогда, в частности,  $f\in W_{loc}^{1,1},$ $f$ --
конечного искажения в $D\setminus\{z_0\},$ кроме того, $f$ дискретно
и открыто. Тогда согласно представлению Стоилова \cite[п.~5 (III),
гл.~V]{St}, $f=\varphi\circ g,$ где $g$ -- некоторый гомеоморфизм, а
$\varphi$ -- аналитическая функция. Заметим, что тогда также $g\in
W_{loc}^{1,1}$ и, кроме того, $g$ имеет конечное искажение.

Действительно, множество точек ветвления $B_{\varphi}\subset
g(D\setminus\{z_0\})$ функции $\varphi$ состоит только из
изолированных точек (см. \cite[пункты 5 и 6 (II), гл.~V]{St}).
Следовательно, $g(z)=\varphi^{-1}\circ f$ локально, вне множества
$g^{-1}\left(B_{\varphi}\right).$ Ясно, что множество
$g^{-1}\left(B_{\varphi}\right)$ также состоит из изолированных
точек, следовательно, $g\in ACL(D\setminus\{z_0\})$ как композиция
аналитической функции $\varphi^{-1}$ и отображения $f\in
W_{loc}^{1,1}(D\setminus\{z_0\}).$

Покажем, что $g\in W_{loc}^{1,1}(D\setminus\{z_0\}).$ Пусть далее
$\mu_f(z)$ означает комплексную дилатацию функции $f(z),$ а
$\mu_g(z)$ -- комплексную дилатацию $g.$ Согласно \cite[(1), п.~C,
гл.~I]{A} для почти всех $z\in D\setminus\{z_0\}$ получаем:
\begin{equation}\label{eq1}
f_z=\varphi_z(g(z))g_z,\qquad
f_{\overline{z}}=\varphi_z(g(z))g_{\overline{z}}\,,
\end{equation}
$\mu_f(z)=\mu_g(z)=:\mu(z), \quad
K_{\mu_f}(z)=K_{\mu_g}(z):=K_{\mu}(z)=\frac{1+|\mu|}{|1-|\mu||}.$
Таким образом,  $K_{\mu}(z)\in L_{loc}^1(D\setminus\{z_0\}).$
Поскольку $f$ -- конечного искажения, из (\ref{eq1}) немедленно
следует, что $g$ также конечного искажения и при почти всех $z\in
D\setminus\{z_0\}$ выполнены соотношения
$|\partial g|\le |\partial g|+ |\overline{\partial} g|=
K^{1/2}_{\mu}(z)(|J(f, z)|)^{1/2},$
откуда по неравенству Гёльдера $|\partial g|\in L_{loc}^1
(D\setminus\{z_0\})$ и $|\overline{\partial} g|\in L_{loc}^1
(D\setminus\{z_0\}).$ Следовательно, $g\in
W_{loc}^{1,1}(D\setminus\{z_0\})$ и $g$ имеет конечное искажение.

В таком случае, $g$ продолжается до гомеоморфизма $g: D\rightarrow
{\Bbb C}$ ввиду \cite[лемма~4 и теорема~4]{Lom}. Тогда $\varphi$
продолжается по непрерывности в точку $g(z_0)$ области $g(D)$ ввиду
классической теоремы Пикара, что и доказывает теорему. $\Box$

\medskip
{\bf 3. Некоторые следствия и замечания.} Ещё один важный результат,
относящийся к устранению особенностей классов Орлича--Соболева,
касается функций конечного среднего колебания (см. \cite{MRSY} и
\cite{IR}).

\medskip
В дальнейшем нам понадобится следующее вспомогательное утверждение
(см., напр., \cite[лемма~7.4, гл.~7]{MRSY} либо
\cite[лемма~2.2]{RS$_1$}).

\medskip
\begin{proposition}\label{pr1A}
{\sl\, Пусть  $x_0 \in {\Bbb R}^n,$ $Q(x)$ -- измеримая по Лебегу
функция, $Q:{\Bbb R}^n\rightarrow [0, \infty],$ $Q\in
L_{loc}^1({\Bbb R}^n).$ Полагаем $A:=A(r_1,r_2,x_0)=\{ x\,\in\,{\Bbb
R}^n : r_1<|x-x_0|<r_2\}$ и
$\eta_0(r)=\frac{1}{Irq_{x_0}^{\frac{1}{n-1}}(r)},$ где
$I:=I=I(x_0,r_1,r_2)=\int\limits_{r_1}^{r_2}\
\frac{dr}{rq_{x_0}^{\frac{1}{n-1}}(r)}$ и
$q_{x_0}(r):=\frac{1}{\omega_{n-1}r^{n-1}}\int\limits_{|x-x_0|=r}Q(x)\,d{\mathcal
H}^{n-1}$ -- среднее интегральное значение функции $Q$ над сферой
$S(x_0, r).$ Тогда
\begin{equation}\label{eq10A}
\frac{\omega_{n-1}}{I^{n-1}}=\int\limits_{A} Q(x)\cdot
\eta_0^n(|x-x_0|)\ dm(x)\le\int\limits_{A} Q(x)\cdot
\eta^n(|x-x_0|)\ dm(x)
\end{equation}
для любой измеримой по Лебегу функции $\eta :(r_1,r_2)\rightarrow
[0,\infty]$ такой, что
$\int\limits_{r_1}^{r_2}\eta(r)dr=1. $ }
\end{proposition}

\medskip Будем говорить, что локально интегрируемая
функция ${\varphi}:D\rightarrow{\Bbb R}$ имеет {\it конечное среднее
колебание} в точке $x_0\in D$, пишем $\varphi\in FMO(x_0),$ если
%
%
%
%
$\limsup\limits_{\varepsilon\rightarrow
0}\frac{1}{\Omega_n\varepsilon^n}\int\limits_{B( x_0,\,\varepsilon)}
|{\varphi}(x)-\overline{{\varphi}}_{\varepsilon}|\, dm(x)<\infty,$
%
%
где
$\overline{{\varphi}}_{\varepsilon}=\frac{1}
{\Omega_n\varepsilon^n}\int\limits_{B(x_0,\,\varepsilon)}
{\varphi}(x)\, dm(x).$
\medskip
Заметим, что, как известно, $\Omega_n\varepsilon^n=m(B(x_0,
\varepsilon)).$ Имеет место следующая

\medskip
\begin{theorem}\label{th3}
{\sl\, Пусть $n\ge 3,$ $x_0\in D,$ тогда каждое открытое, дискретное
и замкнутое ограниченное отображение $f:D\setminus\{x_0\}\rightarrow
{\Bbb R}^n$ класса $W_{loc}^{1, \varphi}(D\setminus\{x_0\})$ с
конечным искажением такое, что $C(f, x_0)\cap C(f,
\partial D)=\varnothing,$ продолжается в точку $x_0$
непрерывным образом до отображения $f:D\rightarrow {\Bbb R}^n,$ если
выполнено условие (\ref{eqOS3.0a}) и, кроме того, найдётся функция
$Q\in L_{loc}^{1}(D),$ такая что $K_I(x, f)\le Q(x)$ при почти всех
$x\in D$ и $Q\in FMO (x_0).$}
\end{theorem}

\medskip
\begin{proof}
Достаточно показать, что условие $Q\in FMO(x_0)$ влечёт расходимость
интеграла (\ref{eq9}), поскольку в этому случае необходимое
заключение будет следовать из теоремы \ref{th1}. Заметим, что для
функций класса $FMO$  в точке $x_0$
\begin{equation}\label{eq31*}
\int\limits_{\varepsilon<|x|<{e_0}}\frac{Q(x+x_0)\, dm(x)}
{\left(|x| \log \frac{1}{|x|}\right)^n} = O \left(\log\log
\frac{1}{\varepsilon}\right)
\end{equation}
при  $\varepsilon \rightarrow 0 $ и для некоторого $e_0>0,$ $e_0 \le
{\rm dist}\,\left(0,\partial D\right).$ При $\varepsilon_0<r_0:={\rm
dist}\,\left(0,\partial D\right)$ полагаем
$\psi(t):=\frac{1}{t\,\log{\frac1t}},$ $I(\varepsilon,
\varepsilon_0):=\int\limits_{\varepsilon}^{\varepsilon_0}\psi(t) dt
=\log{\frac{\log{\frac{1}
{\varepsilon}}}{\log{\frac{1}{\varepsilon_0}}}}$ и
$\eta(t):=\psi(t)/I(\varepsilon, \varepsilon_0).$ Заметим, что
$\int\limits_{\varepsilon}^{\varepsilon_0}\eta(t)dt=1,$ кроме того,
из соотношения (\ref{eq31*}) вытекает, что
\begin{equation}\label{eq32*}
\frac{1}{I^n(\varepsilon,
\varepsilon_0)}\int\limits_{\varepsilon<|x|<\varepsilon_0}
Q(x)\cdot\psi^n(|x|)
 \ dm(x)\le C\left(\log\log\frac{1}{\varepsilon}\right)^{1-n}\rightarrow
 0
 \end{equation}
при $\varepsilon\rightarrow 0.$ Из соотношений (\ref{eq10A}) и
(\ref{eq32*}) вытекает, что интеграл вида (\ref{eq9}) расходится,
что и требовалось установить.
\end{proof} $\Box$

\medskip
Следующее утверждение показывает, что в условиях теорем \ref{th1} и
\ref{th3} требования на функцию $Q$ нельзя, вообще говоря, заменить
условием $Q\in L^p$ ни для какого (сколь угодно большого) $p>0$ и
для любой неубывающей функции $\varphi(t).$ Для простоты рассмотрим
случай, когда $D={\Bbb B}^n,$ $n\ge 3.$

\medskip
\begin{theorem}\label{th3.10.1}{\sl\,
Пусть $\varphi:[0,\infty)\rightarrow[0,\infty)$ -- произвольная
неубывающая функция. Для каждого $p\ge 1$ существуют функция
$Q:{\Bbb B}^n\rightarrow [1, \infty],$ $Q(x)\in L^p({\Bbb B}^n)$ и
равномерно ограниченный гомеоморфизм $g:{\Bbb
B}^n\setminus\{0\}\rightarrow {\Bbb R}^n,$ $g\in W_{loc}^{1,
\varphi}({\Bbb B}^n\setminus\{0\}),$ имеющий конечное искажение,
такой что $K_I(x, f)\le Q(x),$ при этом, $g$ не продолжается по
непрерывности в точку $x_0=0.$}
\end{theorem}

\medskip
\begin{proof} Рассмотрим следующий пример.
Зафиксируем числа $p\ge 1$ и $\alpha\in \left(0, n/p(n-1)\right).$
Можно считать, что $\alpha<1$ в силу произвольности выбора $p.$
Зададим гомеоморфизм $g:{\Bbb B}^n\setminus\{0\}\rightarrow {\Bbb
R}^n$ следующим образом:
$g(x)=\frac{1+|x|^{\alpha}}{|x|}\cdot x.$
Заметим, что отображение $g$ переводит шар $D={\Bbb B}^n$ в кольцо
$D^{\,\prime}=B(0,2)\setminus {\Bbb B}^n,$ при этом, $C(g, 0)={\Bbb
S}^{n-1}$ (отсюда вытекает, что $g$ не имеет предела в нуле).
Заметим, что $g\in C^{1}({\Bbb B}^n\setminus \{0\}),$ в частности,
$g\in W_{loc}^{1,1}.$

\medskip
Далее, в каждой точке $x\in {\Bbb B}^n\setminus \{0\}$ отображения
$g:{\Bbb B}^n\setminus \{0\}\rightarrow {\Bbb R}^n$ вычислим
внутреннюю дилатацию отображения $g$ в точке $x,$ воспользовавшись
правилом (\ref{eq41}). Поскольку $g$ имеет вид
$g(x)=\frac{x}{|x|}\rho(|x|),$ прямым подсчётом соответствующих
производных по направлению можно убедиться, что в качестве главных
векторов $e_{i_1},\ldots, e_{i_n}$ и $\widetilde{e_{i_1}},\ldots,
\widetilde{e_{i_n}}$ можно взять $(n-1)$ линейно независимых
касательных векторов к сфере $S(0, r)$ в точке $x_0,$ где $|x_0|=r,$
и один ортогональный к ним вектор в указанной точке. Соответствующие
главные растяжения (называемые, соответственно, {\it касательными
растяжениями} и {\it радиальным растяжением}) равны
$\lambda_{\tau}(x_0):=\lambda_{i_1}(x_0)=\ldots=\lambda_{i_{n-1}}(x_0)=\frac{\rho(r)}{r}$
и $\lambda_{r}(x_0):=\lambda_{i_n}=\rho^{\,\prime}(r),$
соответственно.

Согласно сказанному, $\lambda_{\tau}(x)=\frac{|x|^{\alpha}+1}{|x|},$
$\lambda_{r}(x)=\alpha|x|^{\alpha-1},$
$l(g^{\,\prime}(x))=\alpha|x|^{\alpha-1},$ $\Vert
g^{\,\prime}(x)\Vert=\frac{|x|^{\alpha}+1}{|x|},$ $|J(x, g)|=
\left(\frac{|x|^{\alpha}+1}{|x|}\right)^{n-1}\cdot
\alpha|x|^{\alpha-1}$ и
$K_I(x, g)=\left(\frac{1+|x|^{\,\alpha}}{\alpha
|x|^{\,\alpha}}\right)^{n-1}.$
Заметим, что если $G$ -- произвольная компактная область в ${\Bbb
B}^n\setminus\{0\},$ то $\Vert g^{\,\prime}(x)\Vert\le c(G)<\infty,$
кроме того, нетрудно видеть, что $|\nabla g(x)|\le n^{1/2}\cdot\Vert
g^{\,\prime}(x)\Vert$ при почти всех $x\in {\Bbb
B}^n\setminus\{0\}.$ Тогда ввиду неубывания функции $\varphi$
выполнено:
$\int\limits_{G} \varphi(|\nabla
g(x)|)dm(x)\le\varphi(n^{1/2}c(G))\cdot m(G)<\infty,$ т.е., $g\in
W^{1, \varphi}(G).$
Заметим, что отображение $g$ имеет конечное искажение, поскольку его
якобиан почти всюду не равен нулю; кроме того, $K_I(x, g)=Q(x),$ где
$Q=\left(\frac{1+|x|^{\,\alpha}}{\alpha
|x|^{\,\alpha}}\right)^{n-1},$ и
$Q(x)\le \frac{C}{|x|^{\alpha(n-1)}}\,,\quad
C:=\left(\frac{2}{\alpha}\right)^{n-1}.$ Таким образом, получаем:
$$\int\limits_{{\Bbb B}^n}\left(Q(x)\right)^p dm(x)\le C^p
\int\limits_{{\Bbb B}^n}\frac{dm(x)}{|x|^{p\alpha(n-1)}}=$$
\begin{equation}\label{eq2.3A}=C^p\int\limits_0^1\int\limits_{S(0,
r)}\frac{d{\mathcal{A}}}{|x|^{p\alpha(n-1)}}\,dr=\omega_{n-1}C^p
\int\limits_0^1\frac{dr}{r^{(n-1)(p\alpha-1)}}\,.\end{equation}
Хорошо известно, что интеграл
$I:=\int\limits_0^1\frac{dr}{r^{\beta}}$ сходится при $\beta<1.$
Таким образом, интеграл в правой части соотношения (\ref{eq2.3A})
сходится, поскольку показатель степени $\beta:=(n-1)(p\alpha-1)$
удовлетворяет условию $\beta<1$ при $\alpha\in (0, n/p(n-1)).$
Отсюда вытекает, что $Q(x)\in L^p({\Bbb B}^n).$
\end{proof}$\Box$

\medskip
Следующее утверждение содержит в себе заключение о том, что условие
(\ref{eq9}) является не только достаточным, но в некотором смысле и
необходимым условием возможности непрерывного продолжения
отображения в изолированную граничную точку.

\medskip
\begin{theorem}\label{th5} {\sl Пусть $\varphi:[0,\infty)\rightarrow[0,\infty)$ -- произвольная
неубывающая функция и $0<\varepsilon_0<1.$ Для каждой измеримой по
Лебегу функции $Q:{\Bbb B}^n\rightarrow [1, \infty],$ $Q\in
L_{loc}({\Bbb B}^n),$ такой, что
$\int\limits_{0}^{\varepsilon_0}\frac{dt}{tq_{0}^{\,\frac{1}{n-1}}(t)}<\infty,$
найдётся ограниченное отображение $f\in W_{loc}^{1, \varphi}({\Bbb
B}^n\setminus\{0\})$ с конечным искажением, которое не может быть
продолжено в точку $x_0=0$ непрерывным образом, при этом, $K_I(x,
f)\le \widetilde{Q}(x)$ п.в., где $\widetilde{Q}(x)$ -- некоторая
измеримая по Лебегу функция, такая что
$\widetilde{q}_0(r):=\frac{1}{\omega_{n-1}r^{n-1}}\int\limits_{S(0,
r)}\widetilde{Q}(x)d{\mathcal H}^{n-1}=q_0(r)$ для почти всех $r\in
(0, 1).$}
\end{theorem}

\begin{proof}
Определим отображение $f:{\Bbb B}^n\setminus\{0\}\rightarrow {\Bbb
R}^n$ следующим образом: $f(x)=\frac{x}{|x|}\rho(|x|),$ где
$\rho(r)=\exp\left\{-\int\limits_{r}^1\frac{dt}{tq_{0}^{1/(n-1)}(t)}\right\}.$
Заметим, что $f\in ACL$ и отображение $f$ дифференцируемо почти
всюду в ${\Bbb B}^n\setminus\{0\}.$ Ввиду техники, изложенной перед
формулировкой леммы \ref{thOS4.1},
$\Vert
f^{\,\prime}(x)\Vert=\frac{\exp\left\{-\int\limits_{|x|}^1\frac{dt}{tq_{0}^{1/(n-1)}(t)}\right\}}{|x|},$
$l(f^{\,\prime}(x))=\frac{\exp\left\{-\int\limits_{|x|}^1\frac{dt}{tq_{0}^{1/(n-1)}(t)}\right\}}
{|x|q_{0}^{1/(n-1)}(|x|)}$ и $|J(x, f)|=\frac{\exp\left\{-n\int
\limits_{|x|}^1\frac{dt}{tq_{0}^{1/(n-1)}(t)}\right\}}{|x|^nq_{0}^{1/(n-1)}(|x|)}.$
Заметим, что $J(x, f)\ne 0$ при почти всех $x,$ значит, $f$ --
отображение с конечным искажением. Кроме того, отметим, что
$\varphi(|\nabla f(x)|)\in L_{loc}^1({\Bbb B}^n\setminus\{0\}),$
поскольку $\Vert f^{\,\prime}(x)\Vert$ локально ограничена в ${\Bbb
B}^n\setminus\{0\},$ а $\varphi$ -- неубывающая функция. Путём
непосредственных вычислений убеждаемся, что $K_I(x, f)=q_{0}(|x|).$
Полагаем $\widetilde{Q}(x):=q_{0}(|x|),$ тогда будем иметь, что
$\widetilde{q}_0(r)=q_0(r)$ для почти всех $r\in (0, 1).$ Наконец,
заметим, что отображение $f$ не может быть продолжено по
непрерывности в точку $x_0=0$ ввиду конструкции этого отображения, а
также условия
$\int\limits_{0}^{\varepsilon_0}\frac{dt}{tq_{0}^{\,\frac{1}{n-1}}(t)}<\infty.$
$\Box$
\end{proof}

\medskip
КОНТАКТНАЯ ИНФОРМАЦИЯ

\medskip
\noindent{{\bf Евгений Александрович Севостьянов} \\
Житомирский государственный университет им.\ И.~Франко\\
кафедра математического анализа, ул. Большая Бердичевская, 40 \\
г.~Житомир, Украина, 10 008 \\ тел. +38 066 959 50 34 (моб.),
e-mail: esevostyanov2009@mail.ru}

\end{document}